%% file: kleenelogicandstructuralcontrol.tex
\documentclass[a4paper, 10pt, english]{article}
\input{packages}

\input{macros2}

\usetikzlibrary{arrows}
\usetikzlibrary{matrix}
\usetikzlibrary{patterns}
\usetikzlibrary{shapes}
\usetikzlibrary{arrows}
\newcolumntype{C}[1]{>{\centering\arraybackslash}p{#1}}
\newcolumntype{L}[1]{>{\arraybackslash}p{#1}}
\DeclareMathSymbol{\Gamma}{\mathalpha}{operators}{0}
\DeclareMathSymbol{\Delta}{\mathalpha}{operators}{1}
\DeclareMathSymbol{\Theta}{\mathalpha}{operators}{2}
\DeclareMathSymbol{\Lambda}{\mathalpha}{operators}{3}
\DeclareMathSymbol{\Xi}{\mathalpha}{operators}{4}
\DeclareMathSymbol{\Pi}{\mathalpha}{operators}{5}
\DeclareMathSymbol{\Sigma}{\mathalpha}{operators}{6}
\DeclareMathSymbol{\Upsilon}{\mathalpha}{operators}{7}
\DeclareMathSymbol{\Phi}{\mathalpha}{operators}{8}
\DeclareMathSymbol{\Psi}{\mathalpha}{operators}{9}
\DeclareMathSymbol{\Omega}{\mathalpha}{operators}{10}
\newtheorem{theorem}{Theorem}
\newtheorem{proposition}{Proposition}
\newtheorem{lemma}{Lemma}

\newtheorem{remark}{Remark}
\newtheorem{definition}{Definition}

\def\mc{\multicolumn}
\newcommand{\fns}{\footnotesize}
\usepackage{authblk}
\EnableBpAbbreviations
\def\fCenter{{\mbox{$\ \vdash\ $}}}

\makeindex

\title{Kleene algebras, adjunction and structural control}

\author[1]{Giuseppe Greco}
\author[2]{Fei Liang}
\author[2,3]{ Alessandra Palmigiano \thanks{This research is supported by the NWO Vidi grant 016.138.314, the NWO Aspasia grant 015.008.054, and a Delft Technology Fellowship awarded to the second author in 2013.}}
\affil[1]{Utrecht University, the Netherlands}
\affil[2]{Delft University of Technology, the Netherlands}
\affil[3]{University of Johannesburg, South Africa}

\date{}

\begin{document}
\maketitle

\begin{abstract}

In the present paper,  we introduce a  multi-type  calculus for the logic of measurable Kleene algebras, for which we prove soundness, completeness, conservativity, cut elimination and subformula property. Our proposal imports ideas and techniques developed in formal linguistics around the notion of structural control \cite{KM}.\\
\\
$\mathbf{Keywords:} $ display calculus, measurable Kleene algebras, structural control.\\
$\mathbf{Math.\ Subject\ Class\ 2010:}$ 03B45, 03G25, 03F05, 08A68.

\end{abstract}

\section{Introduction}
\label{sec:Introduction}
\paragraph{A general pattern.} In this paper, we are going to explore the proof-theoretic ramifications of a pattern which recurs, with different motivations and guises, in various branches of logic,  mathematics, theoretical computer science and formal linguistics. Since the most immediate application we intend to pursue is related to the issue of {\em structural control} in categorial grammar \cite{KM}, we start by presenting this pattern in a way that is amenable to make the connection with structural control.    The pattern we focus on features two {\em types} (of logical languages, of mathematical structures, of data structures, of grammatical behaviour, etc.), a \textsf{General} one and a \textsf{Special} one. Objects of the \textsf{Special} type can be regarded as objects of the \textsf{General} type; moreover,  each \textsf{General} object can be approximated both ``from above'' and ``from below'' by \textsf{Special} objects. That is, there exists a natural notion of order such that the collection of  special objects order-embeds into that of general objects; moreover,  for every general object the smallest special object exists which is greater than or equal to the given general one, and  the greatest special object exists which is smaller than or equal to the given general one.  The situation just described can be captured category-theoretically by stipulating that a given faithful functor  $E: \mathbb{A}\to \mathbb{B}$ between categories $\mathbb{A}$ (of the \textsf{Special} objects) and $\mathbb{B}$ (of the \textsf{General} objects) has both a left adjoint $F: \mathbb{B}\to \mathbb{A}$ and a right adjoint $G: \mathbb{B}\to \mathbb{A}$, and moreover $FE = GE = Id_{\mathbb{A}}$. If we specialize this picture from  categories to posets, the condition above can be reformulated by stating that the order-embedding $e: \mathbb{A}\hookrightarrow \mathbb{B}$ has both a left adjoint $f: \mathbb{B}\twoheadrightarrow \mathbb{A}$ and a right adjoint $g: \mathbb{B}\twoheadrightarrow \mathbb{A}$ such that $fe = ge = id_{\mathbb{A}}$. From these conditions it also follows that the endomorphisms $ef$ and $eg$ on $\mathbb{B}$ are respectively a {\em closure operator} $\gamma: \mathbb{B}\to \mathbb{B}$ (mapping each general object to the smallest special object which is greater than or equal to the given one) and an {\em interior operator } $\iota: \mathbb{B}\to \mathbb{B}$ (mapping each general object to the greatest special object which is smaller than or equal to the given one).

\paragraph{Examples.} A prime example of this situation is the natural embedding map $e$ of the Heyting algebra $\mathbb{A}$ of the up-sets of a poset $\mathbb{W}$, understood as an intuitionistic Kripke structure, into the Boolean algebra $\mathbb{B}$ of the subsets of the domain of the same Kripke structure. This embedding is a complete lattice homomorphism, and hence both its right adjoint and its left adjoint exist. This adjunction situation is the mechanism semantically underlying the celebrated McKinsey-G\"odel-Tarski translation of intuitionistic logic into the classical normal modal logic S4 (cf.~\cite{CPZ:translation} for an extended discussion). Another example arises from the theory of quantales \cite{mulvey1995} (order-theoretic structures arising as ``noncommutative'' generalizations of locales, or pointfree topologies).  For every unital quantale, its two-sided elements\footnote{I.e.~those elements $x$ such that $x\cdot 1\leq x$ and $1\cdot x\leq x$.} form a locale, which is embedded in the quantale, and this embedding has both a left and a right adjoint, so that every element of the quantale is approximated from above and from below by two-sided elements. A third example arises from the algebraic team semantics of inquisitive logic \cite{hodges1997compositional,hodges1997some}, in which the embedding of the algebra interpreting flat formulas into the algebra interpreting general formulas has both a left adjoint and a right adjoint (cf.~\cite{FGPYwollic} for an expanded discussion).

\paragraph{Structural control.} These and other similar adjunction situations provide a promising semantic environment for a line of research in formal linguistics, started in \cite{KM}, and aimed at establishing systematic forms of communication between different grammatical regimes. In \cite{KM}, certain well known extensions of the Lambek calculus are studied as logics for reasoning about the grammatical structure of linguistic resources, in such a way that the requirement of grammatical correctness on the linguistic side is matched by the requirement of derivability on the logical side. In this regard, the various axiomatic extensions of the Lambek calculus correspond to different grammatical regimes which become progressively laxer (i.e.~recognize progressively more constructions as grammatically correct) as their associated logics become progressively stronger. In this context, the basic Lambek calculus incarnates the most general grammatical regime, and the `special' behaviour of its  extensions is captured by additional analytic structural rules. A systematic two-way communication between these grammatical regimes is captured by introducing extra pairs of adjoint modal operators (the {\em structural control} operators), which make it possible to import a degree of flexibility from the special regime into the general regime, and conversely, to endow the special regime with enhanced `structural discrimination' coming from the general regime.
The control operators are normal modal operators inspired by the exponentials of linear logic \cite{Girard} but are not assumed to satisfy the modal S4-type conditions that are satisfied by the linear logic exponentials. Interestingly, in linear logic, precisely the S4-type axioms  guarantee that the `of course'  exponential $\oc$ is an {\em interior operator} and the `why not' exponential $\wn$ is  a {\em closure operator}, and hence each of them can be reobtained as the composition of adjoint pairs of maps between terms of the linear (or general) type and terms of the classical (or special) type, which are section/(co-)retraction pairs. Instead, in \cite{KM},  the adjunction situation is taken as primitive, and 
the structural control adjoint pairs of maps are not section/(co-)retraction pairs. In \cite{GP:linear}, a multi-type environment for linear logic is introduced in which the \textsf{Linear} type encodes the behaviour  of general resources, and the \textsf{Classical/Intuitionistic} type encodes the behaviour of special (renewable) resources. The special behaviour is captured by additional analytic rules (weakening and contraction), and is exported in a controlled form into the general type via the pairs of adjoint connectives which account for the well known controlled application of weakening and contraction in linear logic. This approach has made it possible to design the first calculus for linear logic in which all  rules are closed under uniform substitution (within each type), so that its cut elimination result becomes straightforward. In \cite{GP:linear}   it is also observed that the same underlying mechanisms can be used to account for the controlled application of other structural rules, such as associativity and exchange. Since these are precisely the structural analytic rules capturing the special grammatical regimes in the setting of \cite{KM}, this observation  strengthens the connection between linear logic and the structural control approach of \cite{KM}. 

\paragraph{Kleene algebras: similarities and differences.} In this paper, we focus on the case study of Kleene algebras in close relationship with the ideas of structural control and the multi-type approach illustrated above. Kleene algebras have been introduced to formally capture the behaviour of programs modelled as relations \cite{kozen1994,kozen1997}. While general programs are encoded as arbitrary elements of a Kleene algebra, the Kleene star makes it possible to access the special behaviour of reflexive and transitive programs and to import it in a controlled way within the general  environment. Hence, the role played by the Kleene star is  similar  to the one played by the exponential $\wn$ in linear logic, which  makes it possible to access the special behaviour of renewable resources, captured proof-theoretically by the analytic structural rules of weakening and contraction, and to import it, in a controlled way, into the environment of general resources. 
Another similarity between the Kleene star and $\wn$ is that their axiomatizations guarantee that their algebraic interpretations are closure operators, and hence  can be obtained as the composition of adjoint maps in a way which provides the approximation ``from above'' which is necessary to instantiate the general pattern described above, and use it to justify the soundness of the controlled application of the structural rules capturing the special behaviour. However,  in the general setting of Kleene algebras there is no approximation ``from below'', as e.g.~it is easy to find examples in the context of Kleene algebras of relations in which more than one  reflexive transitive relation can be maximally contained in a given general relation. Our analysis (cf.~Section \ref{sec:Multi-typeLanguage}) identifies the lack of such an approximation ``from below'' as the main hurdle preventing the development of a smooth proof-theoretic treatment of the logic of general Kleene algebras, which to date remains very challenging.

\paragraph{Extant approaches to the logic of Kleene algebras and PDL.} The difficulties in the proof-theoretic treatment of the logic of  Kleene algebras  propagate into the difficulties in the proof-theoretic treatment of Propositional Dynamic Logic (PDL) \cite{Poggiolesi,hartonasanalytic,PDL}. Indeed, PDL can be understood (cf.~\cite{PDL}) as an expansion of the logic of Kleene algebras with a \textsf{Formula} type. Heterogeneous binary operators account for the connection between the  action/program types and the \textsf{Formula} type. The properties of these binary operators are such that their proof-theoretic treatment is per se unproblematic. However, the  PDL axioms encoding the behaviour of the Kleene star are {\em non analytic}, and in the literature several approaches have been proposed to tackle this hurdle, which always involve some trade-off: from sequent calculi with finitary rules but with a non-eliminable analytic cut \cite{hartonasanalytic,jipsen}, to cut-free sequent calculi with infinitary rules \cite{Poggiolesi,palka2007infinitary}.

\paragraph{Measurable Kleene algebras.} In this paper, we introduce a subclass of Kleene algebras, referred to as {\em measurable} Kleene algebras,\footnote{The name is chosen by analogy with measurable sets in analysis, which are defined in terms of the existence of approximations ``from above'' and ``from below''.} which are Kleene algebras endowed with a {\em dual} Kleene star operation, associating any element with its reflexive transitive {\em interior}. Similar definitions have been introduced in the context of dioids (cf.~e.g.~\cite{hardouin2013discrete} and \cite{laurence2014completeness}; in the latter, however, the order-theoretic behaviour of the dual Kleene star is that of a second closure operator rather than that of an interior operator). In measurable Kleene algebras, the defining properties  of the dual Kleene star are those of an {\em interior} operator, which then provides the approximation ``from below'' which is missing in the setting of general Kleene algebras. Hence measurable Kleene algebras are designed to provide yet another instance of the pattern described in the beginning of the present introduction. In this paper, this pattern is used as a semantic support of a proper display calculus for the logic of measurable Kleene algebras, and for establishing a conceptual and technical connection between Kleene algebras and structural control which is potentially beneficial for both areas.

\paragraph{Structure of the paper.} In Section \ref{sec:Preliminaries}, we collect preliminaries on (continuous) Kleene algebras and their logics,  introduce the notion of measurable Kleene algebra, and propose an axiomatization for the logic corresponding to this class. In Section \ref{sec:SemanticEnvironment}, we introduce the heterogeneous algebras corresponding to (continuous, measurable) Kleene algebras and prove that each class of Kleene algebras can be equivalently presented in terms of its heterogeneous counterpart. In Section \ref{sec:Multi-typeLanguage}, we introduce  multi-type languages corresponding to the semantic environments of heterogeneous Kleene algebras, define a translation from the single -type languages to the multi-type languages, and analyze the proof-theoretic hurdles posed by Kleene logic with the lenses of the multi-type environment. This analysis leads to our proposal, introduced in Section \ref{sec:calculus}, of a proper display calculus for the logic of measurable Kleene algebras. In Section \ref{sec:properties} we verify that this calculus is sound, complete, conservative and has cut elimination and subformula property.

\section{Kleene algebras and their logics}
\label{sec:Preliminaries}

\subsection{Kleene algebras and continuous Kleene algebras}
\begin{definition}
\label{def:Ka}
A {\em Kleene algebra} \cite{kozen1990kleene} is a structure $\mathbb{K} = (K, \gor, \cdot, ()^\ast, \gone, \gbot)$ such that:
\begin{enumerate}
\item[K1] $(K, \gor, \gbot)$ is a join-semilattice with bottom element $\gbot$;
\item[K2] $(K, \cdot, \gone)$ is a monoid with  unit $\gone$, moreover $\cdot$ preserves $\gor$ in each coordinate, and $\gbot$ is an annihilator for $\cdot$;
\item[K3] $\gone \gor \alpha \cdot \alpha^\ast \leq  \alpha^\ast$,  $\gone \gor \alpha^\ast \cdot \alpha \leq  \alpha^\ast$, and $\gone \gor \alpha^\ast \cdot \alpha^\ast \leq \alpha^\ast$;
\item[K4] $\alpha \cdot \beta \leq \beta$ implies $\alpha^\ast \cdot \beta\leq \beta$;
\item[K5] $\beta \cdot \alpha \leq \beta$ implies $\beta \cdot \alpha^\ast \leq \beta$.
\end{enumerate}

A Kleene algebra is {\em  continuous} \cite{kozen1990kleene} if:\footnote{For any $n \in \mathbb{N}$ let $\alpha^n$ be defined by induction as follows: $\alpha^0: = \gone$ and $\alpha^{n+1}: = \alpha^n\cdot \alpha$.}
\begin{enumerate}
\item[K1'] $(K, \gor, \gbot)$ is a {\em complete} join-semilattice;
\item[K2'] $\cdot$ is {\em completely} join-preserving in each coordinate;
\item[K6] $\alpha^* = \gOR \alpha^n$ for $n \geq 0$.
\end{enumerate}
\end{definition}

\begin{lemma}\cite[Section 2.1]{kozen1994}
\label{lem:closure}
For any Kleene algebra $\mathbb{K}$ and any $\alpha,\beta \in K$,
\begin{enumerate}
\item $ \alpha \leq \alpha^\ast$;
\item $\alpha^\ast = \alpha^{\ast\ast}$;
\item if $\alpha \leq \beta$,  then $\alpha^\ast \leq \beta^\ast$.
\end{enumerate}
\end{lemma}

By Lemma \ref{lem:closure}, the operation $\ast : K \to K$ is a closure operator on $K$ seen as a poset.
\begin{lemma}
\label{lemma: sufficient for identity ast}
For any continuous Kleene algebra $\mathbb{K}$ and any $\alpha, \beta\in K$,
\[\mbox{If } \alpha\leq \beta \mbox{ and } \gone\leq \beta \mbox{ and } \beta\cdot\beta\leq \beta \mbox{ then } \alpha^\ast \leq \beta.\]
\end{lemma}
Next, we  introduce a subclass of Kleene algebras endowed with both a Kleene star and a dual Kleene star. To our knowledge, this definition has not appeared as such in the literature, although similar definitions have been proposed in different settings (cf.~\cite{laurence2014completeness,BHR}).

\begin{definition}
\label{def:MKa}
A {\em  measurable Kleene algebra} is a structure $\mathbb{K} = (K, \gor, \cdot, ()^\ast, ()^{\star}, \gone, \gbot)$ such that:
\begin{enumerate}
\item[MK1] $(K, \gor, \cdot, ()^\ast,  \gone, \gbot)$ is a continuous Kleene algebra;
\item[MK2] $()^\star$ is a monotone unary operation;

\item[MK3] $\gone\leq \alpha^\star$, and $\alpha^\star\cdot\alpha^\star\leq \alpha^\star$;
 \item[MK4] $\alpha^\star\leq \alpha$ and $\alpha^\star\leq \alpha^{\star\star}$;
 \item[MK5] $\beta\leq \alpha$ and $\gone \leq \beta$ and $\beta\cdot \beta\leq \beta$ implies $\beta\leq \alpha^\star$.
\end{enumerate}
\end{definition}

\begin{lemma}
\label{lemma: ast and star surjective on special}
For any measurable Kleene algebra $\mathbb{K}$ and any $\alpha\in K$, if $\gone\leq \alpha$ and $\alpha\cdot\alpha\leq \alpha$, then
\[\alpha^\ast  = \alpha = \alpha^\star.\]
Hence, \[\mathsf{Range}(\ast ) = \mathsf{Range}(\star ) = \{\beta\in K\mid \gone\leq\beta \mbox{ and } \beta\cdot\beta\leq \beta\}.\]
\end{lemma}
\begin{proof}
By MK4  $\alpha^\star\leq \alpha$; the converse direction follows by MK5 with $\beta: = \alpha$. By Lemma \ref{lem:closure}, $ \alpha\leq  \alpha^\ast$; the converse direction follows from Lemma \ref{lemma: sufficient for identity ast}. This completes the proof of the first part of the statement, and of the inclusion of the set of the $\beta$s with the special behaviour into $\mathsf{Range}(\ast )$ and $\mathsf{Range}(\star )$. The converse inclusions immediately follow from K3 and MK3.
\end{proof}

\subsection{The logics of Kleene algebras}
\label{ssec:kleene logics}
Fix a denumerable set $\mathsf{Atprop}$ of propositional variables, the elements of which are denoted $a, b$ possibly with sub- or superscripts. The language  $\mathsf{KL}$  over  $\mathsf{Atprop}$ is defined recursively as follows:
$$\alpha ::= a \mid \gone \mid \gbot \mid \alpha \gor \alpha \mid \alpha \gand \alpha \mid \alpha^\ast$$
In what follows, we use $\alpha, \beta, \gamma$ (with or without subscripts) to denote formulas in $\mathsf{KL}$.

\begin{definition}
\label{def:KL}
Kleene logic, denoted $\mathrm{S.KL}$, is presented in terms of the following axioms
$$\gbot \fCenter \alpha, \quad \alpha \fCenter \alpha, \quad  \alpha \fCenter \alpha \vee \beta, \quad \beta \fCenter \alpha \vee \beta, \quad \gbot\gand\alpha \dashv\vdash \alpha\gand\gbot, \quad  \gbot\gand\alpha \dashv\vdash \gbot,$$
$$\quad  \gone\gand\alpha \dashv\vdash \alpha\gand\gone, \quad  \gone\gand\alpha \dashv\vdash \alpha, \quad \alpha \gand (\beta \gor \gamma)  \dashv\vdash (\alpha \gand \beta) \gor (\alpha \gand \gamma), \quad (\beta \gor \gamma) \gand  \alpha \dashv\vdash (\beta \gand \alpha) \gor (\gamma \gand \alpha)$$
$$(\alpha \gand \beta) \gand \gamma \dashv\vdash \alpha \gand (\beta \gand \gamma), \quad \gone \gor \alpha \cdot \alpha^\ast \vdash  \alpha^\ast, \quad \gone \gor  \alpha^\ast \cdot \alpha \vdash  \alpha^\ast, \quad \gone \gor \alpha^\ast \cdot \alpha^\ast \vdash \alpha^\ast$$
and the following rules:
\\

\begin{tabular}{lllll}
\AX $\alpha \fCenter \beta$
\AX $\beta \fCenter \gamma$
\BI $\alpha \fCenter \gamma$
\DP
&
\AX $\alpha \fCenter \gamma$
\AX $\beta \fCenter \gamma$
\BI$\alpha \vee \beta \fCenter \gamma$
\DP
&
\AX $\alpha_1 \fCenter \beta_1$
\AX $\alpha_2\fCenter \beta_2$
\BI$\alpha_1 \gand \alpha_2  \fCenter \beta_1\gand \beta_2$
\DP
\end{tabular}
\\
\begin{center}
\begin{tabular}{ll}
\AX$\alpha \gand \beta \fCenter \beta$
\LL{\scriptsize{K4}}
\UI$\alpha^\ast \gand \beta \fCenter \beta$
\DP
&
\AX$ \beta \gand \alpha \fCenter \beta$
\RL{\scriptsize{K5}}
\UI$\beta \gand \alpha^\ast \fCenter \beta$
\DP
\\
\\
\end{tabular}
\end{center}
Continuous Kleene logic, denoted $\mathrm{S.KL_{\omega}}$, is the axiomatic extension of $\mathrm{S.KL}$ determined by the following axioms:
$$\alpha \gand (\gOR_{i \in \omega}\,\, \beta_i)  \dashv\vdash \gOR_{i \in \omega}\,\, (\alpha \gand \beta_i), \quad \gOR_{i \in \omega}\,\, \beta_i \gand \alpha \dashv\vdash \gOR_{i \in \omega}\,\, (\beta_i \gand \alpha),$$
$$ \gOR_{n\geq 0} \,\, \beta\gand\alpha^n\gand\gamma \dashv\vdash \beta\gand\alpha^\ast\gand\gamma$$
\end{definition}

\begin{theorem}\cite{kozen1994}
$(\mathrm{S.KL}_\omega)$ $\mathrm{S.KL}$ is complete with respect to (continuous) Kleene algebras.
\end{theorem}

The language  $\mathsf{MKL}$  over  $\mathsf{Atprop}$ is defined recursively as follows:
$$\alpha ::= a \mid \gone \mid \gbot \mid \alpha \gor \alpha \mid \alpha \gand \alpha \mid \alpha^\ast \mid \alpha^\star.$$

\begin{definition}
\label{def:DualStar}
Measurable Kleene logic, denoted $\mathrm{S.MKL}$,  is presented in terms of the axioms and rules of $\mathrm{S.KL}$ plus the following axioms:
$$\gone \vdash \alpha^\star\quad \alpha^\star\cdot\alpha^\star\vdash \alpha^\star\quad  \alpha^\star\vdash \alpha\quad \alpha^\star\vdash (\alpha^\star)^\star$$
and the following rules:
\begin{center}
\begin{tabular}{ll}
\AX $\alpha \fCenter \beta$
\UI $\alpha^\star \fCenter \beta^\star$
\DP
&
\AX $\beta \fCenter \alpha$
\AX $\gone \fCenter \beta$
\AX $\beta\cdot \beta \fCenter \beta$
\TI $\beta \fCenter \alpha^\star$
\DP
\\
\end{tabular}
\end{center}
\end{definition}


\section{Multi-type semantic environment for Kleene algebras}
\label{sec:SemanticEnvironment}
In the present section, we introduce the algebraic environment which justifies semantically the multi-type approach to the logic of measurable Kleene algebras which we develop in Section \ref{ssec:kleene logics}. In the next subsection, we take Kleene algebras  as starting point, and expand on the properties of the image of the algebraic interpretation of the Kleene star, leading to the notion of `kernel'.  In the remaining  subsections, we show that (continuous, measurable) Kleene algebras  can be equivalently presented in terms of their corresponding heterogeneous algebras.

\subsection{Kleene algebras and their kernels}\label{ssec:KleeneAlgebras}

By Lemma \ref{lem:closure}, for any Kleene algebra $\mathbb{K}$, the operation $()^\ast : K \to K$ is a closure operator on $K$ seen as a poset. By general order-theoretic facts (cf.~\cite[Chapter 7]{DaveyPriestley2002}) this means that  \[()^\ast  =  e\gamma,\] where $\gamma: K\twoheadrightarrow \mathsf{Range}(\ast )$, defined by $\gamma(\alpha) =  \alpha^\ast$ for every $a\in K$, is the left adjoint of the natural embedding  $e: \mathsf{Range}(\ast )\hookrightarrow K$,  i.e.\ for every $\alpha \in K$,  and $\xi \in \mathsf{Range}(\ast)$,
\[\gamma(\alpha)\leq \xi\quad \mbox{ iff }\quad \alpha \leq e(\xi). \]

In what follows, we let  $S$ be the subposet of $K$ identified by $\mathsf{Range}(\ast ) = \mathsf{Range}(\gamma)$. We will also use the variables $\alpha, \beta$, possibly with sub- or superscripts, to denote elements of $K$, and $\pi, \xi, \chi$, possibly with sub- or superscripts, to denote elements of $S$.

\begin{lemma}
For every Kleene algebra $\mathbb{K}$ and every $\xi \in S$,
\begin{equation}\label{lemma:Retraction}
 \gamma(e(\xi)) = \xi.
\end{equation}
\end{lemma}

\begin{proof}
By adjunction, $ \gamma(e(\xi)) \leq \xi$ iff  $e(\xi)\leq e(\xi)$, which always holds. As to the converse inequality $\xi \leq \gamma(e(\xi))$, since $e$ is an order-embedding, it is enough to show that $ e(\xi) \leq e(\gamma(e(\xi)))$, which by adjunction is equivalent to $\gamma(e(\xi))\leq \gamma(e(\xi))$, which always holds.
\end{proof}

\begin{definition}
\label{def:Kernel}
For any  Kleene algebra $\mathbb{K} = (K, \gor, \cdot, ()^\ast, \gone, \gbot)$, let the {\em kernel} of $\mathbb{K}$ be the structure $\mathbb{S} = (S, \tor, \tbot)$  defined as follows:
\begin{itemize}
\item[KK1.] $S: = \mathsf{Range}(\ast ) = \mathsf{Range}(\gamma)$, where $\gamma: K\twoheadrightarrow S$ is defined by letting $\gamma(\alpha) = \alpha^*$ for any $\alpha \in K$;
\item[KK2.] $\xi \tor \chi: = \gamma(e(\xi) \gor e(\chi))$;
\item[KK3.]  $\tbot: = \gamma(\gbot)$.
\end{itemize}
\end{definition}

\begin{proposition}\label{prop:AlgebraicStructureOnKernels}
If $\mathbb{K}$ is a (continuous) Kleene algebra, then its kernel $\mathbb{S}$ defined as above is a (complete) join-semilattice with bottom element.
\end{proposition}

\begin{proof}
By KK1, $S$ is a subposet of $K$. Let $\xi, \chi\in S$. Using KK2 and Lemma \ref{lemma:Retraction}, one shows that $\xi \tor \chi$ is a common upper bound of $\xi$ and $\chi$ w.r.t.~the  order $S$ inherits from $K$. Since $e$ and $\gamma$ are monotone, $\xi\leq \pi$ and $\chi\leq \pi$ imply that  $\xi \tor \chi = \gamma(e(\xi)\gor e(\chi))\leq \gamma (e(\pi)) = \pi$, the last equality due to Lemma \ref{lemma:Retraction}. This shows that $\xi \tor \chi$ is the least upper bound of $\xi$ and $\chi$ w.r.t.~the inherited order. Analogously one shows that, if $\mathbb{K}$ is  continuous  and $Y\subseteq S$, $\bigsqcup Y: = \gamma (\bigcup e[Y])$ is the least upper bound of $Y$. Finally,
$\gamma(\gbot)$  being the bottom element of $S$ follows from $\gbot$ being the bottom element of $K$ and the monotonicity and surjectivity of $\gamma$.
\end{proof}

\begin{remark}\label{rmk:SlightlyMore}
We have proved  a little more than what is stated in  Proposition \ref{prop:AlgebraicStructureOnKernels}. Namely, we have proved that all (finite) joins exist w.r.t.~the order that $S$ inherits from $K$, and hence the join-semilattice structure of $S$ is also in a sense inherited from $K$. However, this does not mean or imply that $S$ is a sub-join-semilattice of $K$, since joins in $S$ are `closures' of joins in $K$, and hence $\tor$ is certainly not the restriction of $\gor$ to $S$.
\end{remark}

\subsection{Measurable Kleene algebras and their kernels}
The results of Section \ref{ssec:KleeneAlgebras} apply in particular to measurable Kleene algebras, where in addition,  by definition, the operation $()^\star : K \to K$ is an interior operator on $K$ seen as a poset. By general order-theoretic facts (cf.~\cite[Chapter 7]{DaveyPriestley2002}) this means that  \[()^\star  =  e'\iota,\] where $\iota: K\twoheadrightarrow \mathsf{Range}(\star )$, defined by $\iota(\alpha) =  \alpha^\star$ for every $a\in K$, is the right adjoint of the natural embedding  $e': \mathsf{Range}(\star )\hookrightarrow K$,  i.e.\ for every $\alpha \in K$  and $\xi \in \mathsf{Range}(\star)$,
\[ e'(\xi)\leq \alpha\quad \mbox{ iff }\quad \xi \leq \iota(\alpha). \]
Moreover, Lemma \ref{lemma: ast and star surjective on special} guarantees that \[\mathsf{Range}(\ast ) = \mathsf{Range}(\star ) = \{\beta\in K\mid \gone\leq\beta \mbox{ and } \beta\cdot\beta\leq \beta\}.\]
Hence, $e'$ coincides with the natural embedding  $e: \mathsf{Range}(\ast )\hookrightarrow K$, which is then endowed with both the left adjoint and the right adjoint.

In what follows, we let  $S$ be the subposet of $K$ identified by \[\mathsf{Range}(\ast ) = \mathsf{Range}(\gamma) = \mathsf{Range}(\iota)  = \mathsf{Range}(\star).\] We will use the variables $\alpha, \beta$, possibly with sub- or superscripts, to denote elements of $K$, and $\pi, \xi, \chi$, possibly with sub- or superscripts, to denote elements of $S$.

\begin{lemma}
\label{lemma:Retraction and coretraction}
For every measurable Kleene algebra $\mathbb{K}$ and every $\xi \in S$,
\begin{equation}
 \gamma(e(\xi)) = \xi = \iota(e(\xi)).
\end{equation}
\end{lemma}
\begin{proof}
The first identity is shown in Lemma \ref{lemma:Retraction}. As to the second one, by adjunction, $\xi\leq \iota(e(\xi))$ iff  $e(\xi)\leq e(\xi)$, which always holds. As to the converse inequality $\iota(e(\xi))\leq \xi$, since $e$ is an order-embedding, it is enough to show that $e(\iota(e(\xi)))\leq  e(\xi) $, which by adjunction is equivalent to $\iota(e(\xi))\leq \iota(e(\xi))$, which always holds.
\end{proof}

\begin{definition}
\label{def:KernelMKA}
For any measurable Kleene algebra $\mathbb{K} = (K, \gor, \cdot, ()^\ast, ()^\star,\gone, \gbot)$, let the {\em kernel} of $\mathbb{K}$ be the structure $\mathbb{S} = (S, \tor, \tbot)$  defined as follows:
\begin{itemize}
\item[KK1.] $S: = \mathsf{Range}(\ast ) = \mathsf{Range}(\gamma) = \mathsf{Range}(\iota)  = \mathsf{Range}(\star)$;
\item[KK2.] $\xi \tor \chi: = \gamma(e(\xi) \gor e(\chi))$;
\item[KK3.]  $\tbot: = \gamma(\gbot)$.
\end{itemize}
\end{definition}

\subsection{Heterogeneous Kleene algebras}
\label{ssec:HeterogeneousKleene}

\begin{definition}
	\label{def:HeterogeneousKleeneAlgebras}
	A {\em heterogeneous Kleene algebra} is a tuple $\mathbb{H} = (\bbA, \bbS, \mand_1, \mand_2, \gamma, e)$ verifying the following conditions:
	
\begin{itemize}
\item[H1] $\bbA = (A, \tor, \cdot, \tone, \gbot)$ is such that $(A, \tor, \gbot)$ a join-semilattice with bottom element $\gbot$ and $(A, \cdot, \tone)$ a monoid with unit $\gone$, moreover $\cdot$ preserves finite joins in each coordinate, and  $\gbot$ is an annihilator for $\cdot$;
\item[H2] $\bbS = (S, \tor, \tbot)$ is a join-semilattice with bottom element $\tbot$;
\item[H3] $\mand_1: \bbS\times \bbA \to \bbA$  preserves finite joins in its second coordinate, is monotone in its first coordinate,  and has unit $\gone$ in its second coordinate, and $\mand_2: \bbA\times \bbS \to \bbA$  preserves finite joins in its first coordinate, is monotone in its second coordinate, and has unit $\gone$ in its first coordinate. Moreover, for all $\alpha\in \bbA$ and $\xi\in S$,
\begin{equation}\label{eq:ThisIsWhereTheBadnessLies}  \xi \mand_1 \alpha =  e(\xi) \cdot \alpha \quad \mbox{ and } \quad \alpha \mand_2 \xi = \alpha \cdot e(\xi);\end{equation}\item[H4] $\gamma: \bbA \twoheadrightarrow \bbS$ and $e: \bbS \hookrightarrow \bbA$ are such that $\gamma\dashv e$ and $\gamma( e(\xi)) = \xi$ for all $\xi \in S$;

\item[H5] $\gone\leq e(\xi)$, and $e(\xi) \cdot e(\xi) \leq e(\xi)$ for any $\xi \in \bbS$;


\item[H6] $\alpha \cdot \beta \leq \beta$ implies $\gamma(\alpha) \mand_1 \beta \leq \beta$,  and  $ \beta \cdot \alpha \leq \beta$ implies $\beta\mand_2  \gamma(\alpha)  \leq \beta$ for all $\alpha, \beta\in \bbA$.
\end{itemize}

A heterogeneous Kleene algebra is {\em continuous} if
\begin{itemize}
\item[H1']  $(A, \tor, \gbot)$ is a {\em complete} join-semilattice and  $\cdot$ preserves {\em arbitrary} joins in each coordinate;
\item[H2'] $\bbS = (S, \tor, \tbot)$ is a {\em complete} join-semilattice;
\item[H7] $e(\gamma(\alpha)) = \gOR\alpha^n$ for any $n \in \mathbb{N}$.
\end{itemize}	
\end{definition}

\begin{definition}
\label{def:Kplus}
For any  Kleene algebra $\mathbb{K} = (K, \gor, \cdot, ()^\ast, \gone, \gbot)$, let \[\mathbb{K}^+ = (\bbA, \bbS,  \mand_1, \mand_2, \gamma, e)\] be the structure defined as follows:
\begin{enumerate}
\item $\bbA := (K, \gor, \cdot, \gone, \gbot)$ is the $()^\ast$-free reduct of $\mathbb{K}$;
\item $\bbS$ is the kernel of $\mathbb{K}$ (cf.~Definition \ref{def:Kernel});
\item $\gamma: \bbA\twoheadrightarrow \bbS$  and $e: \bbS \hookrightarrow \bbA$ are defined as the maps into which the closure operator $()^\ast$ decomposes (cf.~discussion before Lemma \ref{lemma:Retraction});
\item $\mand_1$ (resp.~$\mand_2$) is  the restriction of $\cdot$  to $\mathbb{S}$ in the first (resp.~second) coordinate.
\end{enumerate}
\end{definition}

\begin{proposition}
\label{prop:FromSingleToMulti}
For any (continuous) Kleene algebra $\mathbb{K}$, the structure $\mathbb{K}^+$ defined above is  a (continuous) heterogeneous  Kleene algebra.
\end{proposition}

\begin{proof}
Since $\mathbb{K}$ verifies by assumption  K1 and K2,  $\mathbb{K}^+$ verifies H1. Condition H2 (resp.~H2') is verified by Proposition \ref{prop:AlgebraicStructureOnKernels}. Condition H3 immediately follows from the definition of $\mand_1$ and $\mand_2$ in $\mathbb{K}^+$. Condition H4 holds by Lemma \ref{lem:closure} and \ref{lemma:Retraction}. Condition H5 follows from $\mathbb{K}$ verifying  assumption  K3. Condition H6 follows from $\mathbb{K}$ verifying  assumption  K4 and K5. If $\mathbb{K}$ is continuous, then $\mathbb{K}$ verifies conditions K1', K2' and K6, which guarantee that $\mathbb{K}^+$ verifies H1' and H7.
\end{proof}

\begin{definition}
\label{def:Hplus}
 For any  heterogeneous Kleene algebra $\mathbb{H} = (\bbA, \bbS, \mand_1, \mand_2, \gamma, e)$, let $\mathbb{H}_+: = (\bbA, ()^\ast)$, where  $()^\ast: \bbA\rightarrow \bbA$ is defined by  $\alpha^\ast := e(\gamma(\alpha))$ for every $\alpha \in \bbA$.
\end{definition}

\begin{proposition}
\label{prop:ReverseEngineering}
 For any (continuous) heterogeneous Kleene algebra $\mathbb{H}= (\bbA, \bbS, \mand_1,$ $\mand_2, \gamma, e)$, the structure $\mathbb{H}_+$ defined above is a (continuous) Kleene algebra. Moreover, the kernel of $\mathbb{H}_+$ is join-semilattice-isomorphic  to $\bbS$. 
\end{proposition}

\begin{proof}
As to the first part of the statement, we only need to show that $()^\ast$ satisfies conditions K3-K5 (resp.~K1', K2' and K6) of Definition \ref{def:Ka}.  Condition K3  easily follows from assumption H5 and the proof is omitted. As to K4, let $\alpha, \beta\in \bbA$ such that $\alpha \cdot \beta \leq \beta$.
\begin{center}
\begin{tabular}{r c l l}
 $\alpha \cdot \beta \leq \beta $ & $\Rightarrow$& $\gamma(\alpha) \mand_1 \beta \leq \beta$& (H6)\\
& $\Rightarrow$& $e (\gamma(\alpha)) \cdot \beta \leq \beta$& (H3)\\
& $\Rightarrow$& $\alpha^\ast \cdot \beta \leq \beta$& (definition of $()^\ast$)\\
\end{tabular}
\end{center}
 The proof of K5 is analogous. Conditions K1', K2' and K6 readily follow from assumptions H1' and H7. 

This completes the proof of the first part of the statement. As to the second part, 
let us show preliminarily that the following identities hold:
\begin{itemize}
\item[AK2.] $\xi \tor \chi := \gamma(e(\xi) \gor e (\chi))$ for all $\xi, \chi\in \mathbb{S}$;
\item[AK3.] $\tbot := \gamma(\gbot)$.
\end{itemize}
Being a left adjoint, $\gamma$  preserves existing joins.  Hence, $\gamma(\gbot) = \tbot$, which proves (AK2), and, using H4,  $\gamma(e(\xi) \gor e (\chi)) =  \gamma(e(\xi)) \tor \gamma(e (\chi)) =  \xi \tor \chi$, which proves (AK3).
To show that the kernel of $\mathbb{H}_+$ and $\mathbb{S}$ are isomorphic as  (complete) join-semilattices, notice that the domain of the kernel of $\mathbb{H}_+$ is defined  as  $K_{\ast}: =\mathsf{Range}(()^\ast)$ $=\mathsf{Range}(e\circ \gamma) = \mathsf{Range}(e)$. Since $e$ is an order-embedding (which is easily shown using H4), this implies that $K_{\ast}$, regarded as a sub-poset of $\mathbb{A}$, is order-isomorphic  to the domain of $\mathbb{S}$ with its join-semilattice order. Let $i: \mathbb{S}\to \mathbb{K}_{\ast}$ denote the order-isomorphism between $\mathbb{S}$ and $\mathbb{K}_{\ast}$. To show that $\bbS = (S, \tor_{\bbS}, \tbot)$ and $\mathbb{K}_{\ast} = (K_\ast, \tor_{\mathbb{K}_{\ast}}, 0_{s\ast})$ are isomorphic as join-semilattices, we need to show that for all $\xi, \chi\in \mathbb{S}$,
\[i(\xi \tor_{\mathbb{S}}\chi) = i(\xi)\tor_{\mathbb{K}_{\ast}} i(\chi) \quad\mbox{ and }\quad i(\tbot) = 0_{s\ast}.\] Let $e': \mathbb{K}_{\ast}\hookrightarrow\mathbb{A}$ and $\gamma': \mathbb{A}\twoheadrightarrow\mathbb{K}_{\ast}$ be the pair of  adjoint maps arising from $\ast$.    Thus, $e = e' i$ and $\gamma' = i \gamma$, and so,
\begin{center}
\begin{tabular}{cc ll}
$i(\xi)\tor_{\mathbb{K}_{\ast}} i(\chi)$ & $=$ & $\gamma'(e'(i(\xi)) \gor e'(i(\chi)))$ & (definition of $\tor_{\mathbb{K}_{\ast}}$)\\
&$=$ & $\gamma'(e (\xi) \gor e(\chi))$ & ($e = e' i$)\\
&$=$ & $i(\gamma(e (\xi) \gor e(\chi)))$ & ($\gamma' = i\gamma$)\\
&$=$ & $i(\xi \tor_{\mathbb{S}}\chi)$. & (AK2)\\
&& &\\
$0_{s\ast}$ & $=$ & $ \gamma'(\gbot)$& (KK3) \\
 & $=$  & $i(\gamma(\gbot))$ & $(\gamma' = i \gamma)$ \\
& $=$ & $i(\tbot)$ & (AK3)\\
\end{tabular}
\end{center}
\end{proof}

The following proposition immediately follows from Propositions \ref{prop:FromSingleToMulti} and \ref{prop:ReverseEngineering}:

\begin{proposition}
\label{prop:AplusPlus}
For any Kleene algebra $\mathbb{K}$ and  heterogeneous Kleene algebra $\mathbb{H}$,
\begin{center}
$ \mathbb{K} \cong (\mathbb{K}^+)_+ \quad \mbox{and}\quad \mathbb{H} \cong (\mathbb{H}_+)^+.$
\end{center}
Moreover, these correspondences restrict to continuous Kleene algebras and continuous heterogeneous Kleene algebras.
\end{proposition}

\subsection{Heterogeneous measurable Kleene algebras}
\label{ssec:HeterogeneousKleene meas}
The extra conditions of measurable Kleene algebras allow for their `heterogeneous presentation' (encoded in the definition below) being much simpler than the one for Kleene algebras:
\begin{definition}
	\label{def:HeterogeneousKleeneAlgebras meas}
	A {\em heterogeneous measurable Kleene algebra} is a tuple $\mathbb{H} = (\bbA, \bbS, \iota, \gamma, e)$ verifying the following conditions:
	
\begin{itemize}
\item[HM1] $\bbA = (A, \tor, \cdot, \tone, \gbot)$ is such that $(A, \tor, \gbot)$ a complete join-semilattice with bottom element $\gbot$ and $(A, \cdot, \tone)$ a monoid with unit $\gone$, moreover $\cdot$ preserves arbitrary joins in each coordinate, and  $\gbot$ is an annihilator for $\cdot$;
\item[HM2] $\bbS = (S, \tor, \tbot)$ is a complete join-semilattice with bottom element $\tbot$;
\item[HM3] $e(\gamma(\alpha)) = \gOR\alpha^n$ for any $n \in \mathbb{N}$.
\item[HM4]  $\gamma: \bbA \twoheadrightarrow \bbS$ and $\iota: \bbA \twoheadrightarrow \bbS$ and $e: \bbS \hookrightarrow \bbA$ are such that $\gamma\dashv e\dashv \iota$ and $\gamma( e(\xi)) = \xi = \iota( e(\xi))$ for all $\xi \in S$;
\item[HM5] $\gone\leq e(\xi)$, and $e(\xi) \cdot e(\xi) \leq e(\xi)$ for any $\xi \in \bbS$;
\item[HM6] For any $\beta\in \bbA$, if $\gone\leq \beta$ and $\beta \cdot \beta \leq \beta$, then  $\gamma (\beta)\leq \iota(\beta)$.
\end{itemize}
\end{definition}

\begin{definition}
\label{def:Kplus meas}
For any measurable Kleene algebra $\mathbb{K} = (K, \gor, \cdot, ()^\ast, ()^\star, \gone, \gbot)$, let \[\mathbb{K}^+ = (\bbA, \bbS,  \iota, \gamma, e)\] be the structure defined as follows:
\begin{enumerate}
\item $\bbA := (K, \gor, \cdot, \gone, \gbot)$ is the $\{()^\ast, ()^\star\}$-free reduct of $\mathbb{K}$;
\item $\bbS$ is the kernel of $\mathbb{K}$ (cf.~Definition \ref{def:KernelMKA});
\item $\gamma: \bbA\twoheadrightarrow \bbS$  and $e: \bbS \hookrightarrow \bbA$ are defined as the maps into which the closure operator $()^\ast$ decomposes, and $\iota: \bbA\twoheadrightarrow \bbS$  and $e: \bbS \hookrightarrow \bbA$ are defined as the maps into which the interior operator $()^\star$ decomposes (cf.~discussion before Lemma \ref{lemma:Retraction and coretraction}).
\end{enumerate}
\end{definition}

\begin{proposition}
\label{prop:FromSingleToMulti meas}
For any measurable Kleene algebra $\mathbb{K}$, the structure $\mathbb{K}^+$ defined above is  a  heterogeneous measurable Kleene algebra.
\end{proposition}
\begin{proof}
Since $\mathbb{K}$ verifies by assumption  K1', K2, and K6,  $\mathbb{K}^+$ verifies HM1. Condition HM2  is verified by Proposition \ref{prop:AlgebraicStructureOnKernels}. Condition HM3 immediately follows from the definition of $()^\ast$ and assumption K6. Condition HM4 holds by Lemma \ref{lemma:Retraction and coretraction}. Condition HM5 follows from $\mathbb{K}$ verifying  assumption  K3.
As to condition HM6, if $\gone \leq \beta$ and $\beta\cdot \beta\leq \beta$, then by Lemma \ref{lemma: ast and star surjective on special}, $e(\gamma(\beta)) = \beta^\ast =  \beta^\star = e(\iota(\beta))$, which implies, since $e$ is injective, that $\gamma(\beta)\leq \iota(\beta)$, as required.
\end{proof}

\begin{definition}
\label{def:Hplus}
 For any  heterogeneous measurable Kleene algebra $\mathbb{H} = (\bbA, \bbS, \iota, \gamma, e)$, let $\mathbb{H}_+: = (\bbA, ()^\ast, ()^\star)$, where  $()^\ast: \bbA\rightarrow \bbA$ and $()^\star: \bbA\rightarrow \bbA$ are respectively defined by  $\alpha^\ast := e(\gamma(\alpha))$ and $\alpha^\star := e(\iota(\alpha))$ for every $\alpha \in \bbA$.
\end{definition}

\begin{proposition}
\label{prop:ReverseEngineering meas}
 For any  heterogeneous measurable Kleene algebra $\mathbb{H}= (\bbA, \bbS, \iota, \gamma, e)$, the structure $\mathbb{H}_+$ defined above is a measurable Kleene algebra. Moreover, the kernel of $\mathbb{H}_+$ is join-semilattice-isomorphic  to $\bbS$. 
\end{proposition}

\begin{proof}
The part of the statement which concerns the verification of axioms K1', K2', K3-K6 is accounted for as in the proof of Proposition \ref{prop:ReverseEngineering}. Let us verify that $()^\star$ satisfies conditions MK2-MK5 of Definition \ref{def:MKa}.  Conditions MK2 and MK4  easily follow from the assumption that $e\dashv \iota$ (HM4). Condition MK3 follows from the surjectivity of $\iota$ and assumption HM5. As to MK5, it is enough to show that if $\alpha, \beta\in K$ such that $\beta\leq \alpha$ and $\gone \leq \beta$ and $\beta\cdot \beta\leq \beta$, then $\beta\leq e(\iota(\alpha))$. Since $\beta\leq \alpha$ by assumption and $e$ and $\iota$ are monotone, it is enough to show that $\beta\leq e(\iota(\beta))$. By adjunction, this is equivalent to $\gamma(\beta)\leq \iota(\beta)$, which holds by assumption HM6. This completes the proof of the first part of the statement. The proof of the second part is analogous to the corresponding part of the proof of Proposition \ref{prop:ReverseEngineering}, and is omitted.
\end{proof}

The following proposition immediately follows from Propositions \ref{prop:FromSingleToMulti meas} and \ref{prop:ReverseEngineering meas}:

\begin{proposition}
\label{prop:AplusPlus meas}
For any measurable Kleene algebra $\mathbb{K}$ and  heterogeneous measurable Kleene algebra $\mathbb{H}$,
\begin{center}
$ \mathbb{K} \cong (\mathbb{K}^+)_+ \quad \mbox{and}\quad \mathbb{H} \cong (\mathbb{H}_+)^+.$
\end{center}
\end{proposition}

\section{Multi-type  presentations for Kleene logics}\label{sec:Multi-typeLanguage}
In Section \ref{ssec:HeterogeneousKleene}, (continuous)  heterogeneous (measurable) Kleene algebras  have been introduced (cf.~Definitions  \ref{def:HeterogeneousKleeneAlgebras} and \ref{def:HeterogeneousKleeneAlgebras meas})  and shown to be  equivalent presentations of (continuous, measurable)  Kleene algebras. These constructions  motivate the  multi-type presentations of Kleene logics we introduce in the present section. Indeed, heterogeneous Kleene algebras are natural models for the following multi-type language $\mathcal{L}_{\mathrm{MT}}$, defined by simultaneous induction  from a  set $\mathsf{AtAct}$ of atomic actions (the elements of which are denoted by letters $a, b$):
\begin{align*}
\mathsf{Special}\ni \xi ::=&\, \gtdia \alpha   \\
\mathsf{General}\ni  \alpha ::= & \,a \mid 1 \mid 0 \mid \tgbox \xi  \mid \alpha \gor \alpha 
\end{align*}

while heterogeneous measurable Kleene algebras are natural models for the following multi-type language $\mathcal{L}_{\mathrm{MT}}$, defined by simultaneous induction  from $\mathsf{AtAct}$:
\begin{align*}
\mathsf{Special}\ni \xi ::=&\, \gtdia \alpha \mid \gtbox \alpha  \\
\mathsf{General}\ni  \alpha ::= & \,a \mid 1 \mid 0 \mid \tgbox \xi  \mid \alpha \gor \alpha 
\end{align*}
where, in any heterogeneous (measurable) Kleene algebra, the maps $\gamma$ and $e$ (and $\iota$) interpret the heterogeneous connectives  $\gtdia$, $\tgbox$ (and $\gtbox$) respectively.
The interpretation of $\mathcal{L}_{\mathrm{MT}}$-terms  into heterogeneous  algebras  is defined as the straightforward generalization of the interpretation of propositional languages in algebras of compatible signature, and is omitted.

The toggle between Kleene algebras  and heterogeneous Kleene algebras  is reflected syntactically by the following translation $(\cdot)^t: \mathcal{L}\to \mathcal{L}_{\mathrm{MT}}$ 
between the original language $\mathcal{L}$ of Kleene logic  and  the language $\mathcal{L}_{\mathrm{MT}}$ defined above:
\begin{center}
\begin{tabular}{r c l }
$a^t$ &$ = $& $a$ \\
$1^t$ &$ = $& $1$ \\
$0^t$ &$ = $& $0$ \\
$(\alpha\gor \beta)^t$ &$ = $& $\alpha^t \gor \beta^t$\\
$(\alpha\cdot \beta)^t$ &$ = $& $\alpha^t \cdot \beta^t$\\
$(\alpha^\ast)^t$ &$ = $&$\tgbox\gtdia \alpha^t$\\
$(\alpha^\star)^t$ &$ = $&$\tgbox\gtbox \alpha^t$\\
\end{tabular}
\end{center}
The following proposition is proved by a routine induction on  $\mathcal{L}$-formulas.
\begin{proposition}
\label{prop:ConsequencePreservedAndReflected}
For all $\mathcal{L}$-formulas $A$ and $B$ and every Kleene algebra $\mathbb{K}$,
\[\mathbb{K}\models \alpha\leq \beta \quad \mbox{ iff }\quad \mathbb{K}^+\models \alpha^t\leq \beta^t.\]
\end{proposition}

The general definition of {\em analytic inductive} inequalities can be instantiated to inequalities in the $\mathcal{L}_{\mathrm{MT}}$-signature according to the order-theoretic properties of the algebraic interpretation of the $\mathcal{L}_{\mathrm{MT}}$-connectives in heterogeneous (measurable) Kleene algebras. In particular, all connectives but $\mand_1$ and $\mand_2$ are normal.   Hence, we are now in a position to translate the axioms and rules describing the behaviour of $()^\ast$ and $()^\star$ from the single-type languages  into $\mathcal{L}_{\mathrm{MT}}$ using $(\cdot)^t$, and verify whether the resulting translations are analytic inductive.

\begin{center}
\begin{tabular}{r l}
$\gone \gor \alpha \leq \alpha^\ast\ \rightsquigarrow $ &
  $\begin{cases}
   \gone \gor \alpha^t \leq \tgbox\gtdia \alpha^t & (i)\\
  \tgbox\gtdia \alpha^t \leq \gone \gor \alpha^t& (ii)
  \end{cases}
  $
  \\
  \\
$\gone \gor \alpha^\ast = \alpha^\ast \ \rightsquigarrow$ &
$ \begin{cases}
    \gone \gor \tgbox\gtdia \alpha^t  \leq \tgbox\gtdia \alpha^t & (iii)\\
  \tgbox\gtdia \alpha^t \leq \gone \gor \tgbox\gtdia \alpha^t  & (iv)
  \end{cases}
  $
  \\
  \\
  $
\alpha \cdot \beta \leq \beta$ implies $\alpha^\ast \cdot \beta \leq \beta  \rightsquigarrow $&
  $\begin{cases}
\alpha^t \cdot \beta^t \leq \beta^t$ implies $\tgbox\gtdia \alpha^t \cdot \beta^t \leq \beta^t& (v)\\
  \end{cases}
  $
  \\
\\
$
\beta \cdot \alpha \leq \beta$ implies $\beta \cdot \alpha^\ast \leq \beta  \rightsquigarrow $&
  $\begin{cases}
\beta^t \cdot \alpha^t \leq \beta^t$ implies $\beta^t \cdot \tgbox\gtdia \alpha^t  \leq \beta^t& (vi)\\
  \end{cases}
  $

  \end{tabular}
  \end{center}
Notice that, relative to the order-theoretic properties of their interpretations on heterogeneous Kleene algebras, $\cdot$, $\gone$, $\gtdia$  are $\mathcal{F}$-connectives, while $\tgbox$ is a $\mathcal{G}$-connective. However, relative to the order-theoretic properties of their interpretations on heterogeneous measurable Kleene algebras, $\cdot$, $\gone$, $\gtdia$  are $\mathcal{F}$-connectives, while $\tgbox$ is both an $\mathcal{F}$-connective and a  $\mathcal{G}$-connective.
Hence, it is easy to see that, relative to the first interpretation, $(i)$ is the only analytic inductive inequality of the list above, due to the occurrences of the McKinsey-type nesting $\tgbox\gtdia \alpha^t$ in antecedent position. However, relative to the second interpretation, the same nesting becomes harmless, since the occurrences of $\tgbox$ in antecedent position are part of the Skeleton.

Likewise, it is very easy to see that the conditions HM1-HM6 in the definition of heterogeneous measurable Kleene algebras do not violate the conditions on nesting of analytic inductive inequalities. However, some of these conditions do not consist of inequalities taken in isolation but are given in the form of quasi-inequalities. When embedded into a quasi-inequality, the proof-theoretic treatment of an inequality such as $\beta\cdot\beta\leq \beta$ (which in isolation would be unproblematic) becomes problematic, since the translation of the quasi-inequality into a logically equivalent rule would not allow to `disentangle' the occurrences of $\beta$ in precedent position from the occurrences of $\beta$ in succedent position, thus making it impossible to translate the quasi-inequality directly as an analytic structural rule. This is why the calculus defined in the following section features an infinitary rule, introduced to circumvent this problem.   

\section{The proper multi-type display calculus D.MKL}
\label{sec:calculus}
\subsection{Language}
\label{sec:LanguageAndRules}

   In the present section, we define a {\em multi-type language} for the proper  multi-type display calculus for measurable Kleene logic. As usual, this language includes constructors for both logical (operational) and structural terms.
\begin{itemize}
\item Structural and operational terms:
\end{itemize}

\begin{center}
\begin{tabular}{ll}
$\mathsf{General}$&$ \left\{\begin{array}{l}
\alpha ::= \,a \mid \gone \mid \gbot \mid \tgbox\xi \mid \alpha \gor \alpha \mid \alpha \gand \alpha\\ 
 \\
 \Gamma ::= \gONE \mid \WCIRC\Pi\mid \Gamma \gAND \Gamma \mid \Gamma < \Gamma \mid \Gamma > \Gamma \\ 
\end{array} \right.$\\
\\

$\mathsf{Special}$& $\left\{\begin{array}{l}
\xi ::= \, \gtdia\alpha \mid \gtbox\alpha \\ 
 \\
\Pi ::= \BCIRC\Gamma
\end{array} \right.$
\end{tabular}
\end{center}

In what follows, we reserve $\alpha, \beta, \gamma$ (with or without subscripts) to denote $\mathsf{General}$-type operational terms, and $\xi, \chi, \pi$ (with or without subscripts) to denote formulas in $\mathsf{Special}$-type operational terms. Moreover, we reserve $\Gamma, \Delta, \Theta$ (with or without subscripts) to denote $\mathsf{General}$-type structural terms, and $\Pi, \Xi, \Lambda$ (with or without subscripts) to denote  $\mathsf{Special}$-type structural terms.

\begin{itemize}
\item Structural and operational terms:
\end{itemize}

\begin{center}
\begin{tabular}{|c|c|c|c|c|c|c|c|c|c|c|c|}
\hline
\mc{8}{|c|}{$\mathsf{General}$} &  \mc{2}{c|}{$\mathsf{S} \to \mathsf{G}$} & \mc{2}{c|}{$\mathsf{G} \to \mathsf{S}$}\\
\hline
\mc{2}{|c|}{$\gBOT$}& \mc{2}{c|}{$\gAND$} & \mc{2}{c|}{$<$}& \mc{2}{c|}{$>$} & \mc{2}{c|}{$\WCIRC$} & \mc{2}{c|}{$\BCIRC$}  \\
\hline
$\gone$&$\gbot$    & $\gand$  &$\phantom{\cdot}$ & $\phantom{(\slash)}$ & $(\slash)$ & $\phantom{(\backslash)}$ & $(\backslash)$ & $\tgbox$ & $\tgbox$ & $\gtdia$  &$\gtbox$ \\
\hline
\end{tabular}
\end{center}

Notice that, for the sake of minimizing the number of structural symbols,  we are assigning the same structural connective $\BCIRC$ to $\gtdia$  and $\gtbox$  although these modal operators are not dual to one another, but are respectively interpreted as the left adjoint and the right adjoint of $\tgbox$, which is hence both an $\mathcal{F}$-operator and a $\mathcal{G}$-operator, and  can therefore correspond to the structural connective $\WCIRC$ both in antecedent and in succedent position.

\subsection{Rules}
In the rules below, the symbols $\Gamma, \Delta$ and $\Theta$ denote structural variables of general type, and $\Sigma, \Pi$ and $\Xi$ structural variables of special type. The calculus D.MKL consists the following rules:
\begin{itemize}
\item Identity and cut rules:
\begin{center}
\begin{tabular}{rl}
\mc{2}{c}{
\AxiomC{\phantom{$X \fCenter A$}}
\LeftLabel{\scriptsize Id}
\UI $a \fCenter a$
\DisplayProof
}
\\
\\
\AX $\Gamma \fCenter \alpha$
\AX $\alpha \fCenter \Delta$
\RightLabel{\scriptsize $\mathsf{Cut}_g$}
\BI$\Gamma \fCenter \Delta$
\DisplayProof
&
\AX $\Pi \fCenter \xi$
\AX $\xi \fCenter \Xi$
\RightLabel{\scriptsize $\mathsf{Cut}_s$}
\BI$\Pi \fCenter \Xi$
\DisplayProof
\\
\end{tabular}
\end{center}

\item $\mathsf{General}$ type display rules:
\begin{center}
\begin{tabular}{rl}
\AX $\Gamma \gAND \Delta \fCenter \Theta$
\LeftLabel{\scriptsize $\mathsf{res}$}
\doubleLine
\UI$\Delta \fCenter \Gamma > \Theta$
\DisplayProof
&
\AX $\Gamma \gAND \Delta \fCenter \Theta $
\RightLabel{\scriptsize $\mathsf{res}$}
\doubleLine
\UI$\Gamma \fCenter \Theta < \Delta$
\DisplayProof
\end{tabular}
\end{center}

\item Multi-type display rules:
\begin{center}
\begin{tabular}{rl}
\AX $\Gamma \fCenter \WCIRC\Xi$
\LeftLabel{\scriptsize $\mathsf{adj}$}
\doubleLine
\UI$\BCIRC\Gamma \fCenter \Xi$
\DisplayProof
&
\AX $\WCIRC\Xi \fCenter \Gamma$
\RL{\scriptsize $\mathsf{adj}$}
\doubleLine
\UI$\Xi \fCenter \BCIRC\Gamma$
\DisplayProof
\end{tabular}
\end{center}

\item $\mathsf{General}$ type structural rules:

\begin{center}
\begin{tabular}{rl}
\AX $\Gamma \fCenter \Delta$
\LL{\scriptsize $\gONE_L$}
\doubleLine
\UIC{$\gONE \gAND \Gamma \fCenter \Delta$}
\DisplayProof
 &
\AX $\Gamma \fCenter \Delta$
\RL{\scriptsize $\gONE_R$}
\doubleLine
\UIC{$\Gamma \gAND \gONE \fCenter \Delta$}
\DisplayProof

\\
\\
\AXC{$(\Gamma_1 \gAND \Gamma_2) \gAND \Gamma_3 \fCenter \Delta$}
\LL{\fns assoc}
\doubleLine
\UIC{$\Gamma_1 \gAND (\Gamma_2 \gAND \Gamma_3) \fCenter \Delta$}
\DP
&
\AX$\Gamma \fCenter \gBOT$
\RightLabel{\scriptsize $\gBOT$-W}
\UIC{$\Gamma \fCenter \Delta$}
\DisplayProof
 \\
\end{tabular}
\end{center}

\item Multi-type structural rules:\footnote{Let $\Gamma^{(n)}$  be defined by setting $\Gamma^{(1)}: = \Gamma$ and $\Gamma^{(n+1)}: = \Gamma\odot \Gamma^{(n)}$.}

\begin{center}
\begin{tabular}{rl}
\AXC{}
\LL{\scriptsize one}
\UI$\gONE \fCenter \WCIRC\Pi$
\DisplayProof
&
\AX$\Gamma \fCenter \WCIRC\Pi$
\AX$\Delta \fCenter \WCIRC\Pi$
\RightLabel{\scriptsize abs}
\BI$\Gamma \odot \Delta \fCenter \WCIRC\Pi$
\DisplayProof
\\
\\
\AX$\Pi \fCenter \Sigma$
\LL{\scriptsize b-bal}
\UI$ \BCIRC\WCIRC \Pi \fCenter \BCIRC\WCIRC\Sigma$
\DisplayProof
 &
\AX$ \Pi \fCenter \Xi$
\RL{\scriptsize w-bal}
\doubleLine
\UI$\WCIRC\Pi \fCenter \WCIRC\Xi$
\DisplayProof
\\
\\

\AXC{$( \Gamma^{(n)}  \fCenter \Delta \,\mid \,n \geq 1)$}
\LL{\scriptsize $\omega$}
\UIC{$\WCIRC \BCIRC \Gamma  \fCenter \Delta $}
\DisplayProof

&
\AX$\WCIRC \Pi \gAND \WCIRC\Pi \fCenter \Delta$
\RL{\fns $\WCIRC$-C}
\UI$\WCIRC\Pi \fCenter \Delta$
\DP

 \\
\end{tabular}
\end{center}

\item  $\mathsf{General}$ type operational rules: in what follows, $i \in \{1, 2\}$,
\begin{center}
\begin{tabular}{rl}
\AX$\gONE \fCenter \Delta$
\LL{\scriptsize $\gone$}
\UI$\gone \fCenter \Delta$
\DP
&
\AXC{}
\RL{\scriptsize $\gone$}
\UI$\gONE \fCenter \gone$
\DP
\\
\\
\AXC{}
\LL{\scriptsize $\gbot$}
\UI$\gbot \fCenter \gBOT$
\DP
&
\AX$\Gamma \fCenter \gBOT$
\RL{\scriptsize $\gbot$}
\UI$\Gamma \fCenter \gbot$
\DP
\\
\\
\AX$\alpha_1 \fCenter \Delta$
\AX$\alpha_2 \fCenter \Delta$
\LL{\scriptsize $\gor$}
\BI$\alpha_1 \gor \alpha_2 \fCenter \Delta$
\DP
&
\AX$\Gamma \fCenter \alpha_{i}$
\RL{\scriptsize $\gor$}
\UI$\Gamma \fCenter \alpha_1 \gor \alpha_2$
\DP
\\
\\
\AX$\alpha \odot \beta \fCenter \Delta$
\LL{\scriptsize $\cdot$}
\UI$\alpha \cdot \beta \fCenter \Delta$
\DP
&
\AX$\Gamma \fCenter \alpha$
\AX$\Delta \fCenter \beta$
\RL{\scriptsize $\cdot$}
\BI$\Gamma \odot \Delta \fCenter\alpha \cdot \beta$
\DP

\end{tabular}
\end{center}

\item Multi-type operational rules:

\begin{center}
\begin{tabular}{rl}
\AX$\BCIRC\alpha \fCenter \Pi$
\LL{\scriptsize $\gtdia$}
\UI$\gtdia\alpha \fCenter \Pi$
\DisplayProof
&
\AX$\Gamma \fCenter \alpha$
\RL{\scriptsize $\gtdia$}
\UI$\BCIRC\Gamma \fCenter \gtdia\alpha$
\DisplayProof
\\

\\

\AX$\alpha \fCenter \Gamma$
\LL{\scriptsize $\gtbox$}
\UI$\gtbox\alpha \fCenter \BCIRC\Gamma$
\DisplayProof
&
\AX$\Pi \fCenter \BCIRC \alpha$
\RL{\scriptsize $\gtbox$}
\UI$\Pi \fCenter \gtbox \alpha$
\DisplayProof
\\

\\

\AX$\WCIRC \xi \fCenter \Gamma$
\LL{\scriptsize $\tgbox$}
\UI$\tgbox\xi  \fCenter \Gamma$
\DisplayProof
&
\AX$\Gamma \fCenter \WCIRC\xi$
\RL{\scriptsize $\tgbox$}
\UI$\Gamma \fCenter \tgbox\xi$
\DisplayProof
\end{tabular}
\end{center}
\end{itemize}

The following fact is proven by a straightforward induction on $\alpha$ and $\xi$. We omit the details.
\begin{proposition}
\label{prop:identity lemma}
For every $\alpha \in \textsf{General}$ and $\xi \in \textsf{Special}$, the sequents $\alpha \fCenter \alpha$ and $\xi\fCenter \xi$ are derivable in D.MKL.
\end{proposition}

\section{Properties}\label{sec:properties}
\subsection{Soundness}
In the present subsection, we outline the  verification of the soundness of the rules of $\mathrm{D.MKL}$ w.r.t.~heterogenous measurable Kleene algebras  (cf.~Definition \ref{def:HeterogeneousKleeneAlgebras meas}). The first step consists in interpreting structural symbols as logical symbols according to their (precedent or succedent) position, as indicated in the synoptic table of Section \ref{ssec:Display calculus}. This makes it possible to interpret sequents as inequalities, and rules as quasi-inequalities. For example, (modulo standard manipulations) the rules on the left-hand side below correspond to the (quasi-)inequalities on the right-hand side:

\begin{center}
\begin{tabular}{rcl}
\AXC{}
\UI$\Phi  \fCenter \WCIRC\Pi$
\DisplayProof
&$\quad\rightsquigarrow\quad$&
$\forall\xi[1 \leq \tgbox\xi]$
\\
\\
\AX$\Gamma \fCenter \WCIRC\Pi$
\AX$\Delta \fCenter \WCIRC\Pi$
\RightLabel{\scriptsize abs}
\BI$\Gamma \odot \Delta \fCenter \WCIRC\Pi$
\DisplayProof
&$\quad\rightsquigarrow\quad$& $\forall\alpha\forall\beta[\gtdia(\alpha\cdot\beta) \leq \gtdia\alpha\sqcup \gtdia\beta]$

\\
\\
\AX$\Pi \fCenter \Sigma$
\LL{\scriptsize b-bal}
\UI$ \BCIRC\WCIRC \Pi \fCenter \BCIRC\WCIRC\Sigma$
\DisplayProof
&$\quad\rightsquigarrow\quad$& $\forall\xi[\gtdia\tgbox\xi \leq \gtbox\tgbox\xi]$
\\
\\
\AX$ \Pi \fCenter \Xi$
\RL{\scriptsize w-bal}
\doubleLine
\UI$\WCIRC\Pi \fCenter \WCIRC\Xi$
\DisplayProof
&$\quad\rightsquigarrow\quad$& $\forall\xi\forall \pi[\pi\leq \xi \Leftrightarrow \tgbox\pi \leq \tgbox\xi]$
\\
\\

\AXC{$( \Gamma^{(n)}  \fCenter \Delta \,\mid \,n \geq 1)$}
\LL{\scriptsize $\omega$}
\UIC{$\WCIRC \BCIRC \Gamma  \fCenter \Delta $}
\DisplayProof
&$\quad\rightsquigarrow\quad$& $\forall\alpha[\tgbox\gtdia\alpha \leq \bigcup_{n\in \omega}\alpha^n]$
\\
\\
\AX$\WCIRC \Pi \gAND \WCIRC\Pi \fCenter \Delta$
\RL{\fns $\WCIRC$-C}
\UI$\WCIRC\Pi \fCenter \Delta$
\DP
&$\quad\rightsquigarrow\quad$& $\forall\xi[\tgbox\xi \leq \tgbox\xi\cdot \tgbox\xi]$

\\
\end{tabular}
\end{center}
Then, the verification of the soundness of the rules of $\mathrm{D.MKL}$ boils down to checking the validity of their corresponding quasi-inequalities in heterogenous measurable Kleene algebras. This verification  is routine and is omitted.
\subsection{Completeness}
In the present section, we show that the translations --  by means of the map $()^t$ defined in Section \ref{sec:Multi-typeLanguage} -- of the axioms and rules of  $\mathrm{S.MKL}$ (cf.~Section \ref{ssec:kleene logics}) are derivable in the calculus $\mathrm{D.MKL}$. For the reader's convenience, here below we report the recursive definition of $()^t$:
\begin{center}
\begin{tabular}{rclcrcl}
$a^t$   & $::=$ & $a$\\
$\gone^t$ & $::=$ & $\gone$ \\
$\gbot^t$ & $::=$ & $\gbot$\\
$(\alpha \gand \beta)^t$ & $::= $ & $\alpha^t\gand \beta^t$\\
$(\alpha \gor \beta)^t$ & $::=$ & $ \alpha^t \gor \beta^t$ \\
$(\alpha^{\ast})^t$ & $::=$ & $\tgbox\gtdia\alpha^t$\\
$(\alpha^{\star})^t$ & $::=$ & $\tgbox\gtbox\alpha^t$ \\
\end{tabular}
\end{center}

\begin{proposition}
For every $\alpha \in \mathrm{S.KL}$, the sequent $\alpha^t \fCenter \alpha^t$ is derivable in D.MKL.
\end{proposition}

Let $\alpha^{(n)}$ be defined by setting $\alpha^{(1)}: = \alpha$ and $\alpha^{(n+1)}: = \alpha\odot \alpha^{(n)}$.

\begin{lemma}[Omega]
\label{lemma:OmegaLemma}
If $\alpha \odot \beta \fCenter \beta$ (resp.~$\beta \odot \alpha \fCenter \beta$) is derivable, then $\alpha^{(n)} \odot \beta \fCenter \beta$ (resp.~$\beta \odot \alpha^{(n)} \fCenter \beta$) is derivable for every $n\geq 0$.
\end{lemma}

\begin{proof}
Let us show that for any $n\geq 1$, if $\alpha^{(n)} \odot \beta \fCenter \beta$ is derivable, then $\alpha^{(n+1)} \odot \beta \fCenter \beta$ is derivable (the proof that $\beta \odot \alpha^{(n)} \fCenter \beta$ is derivable from $\beta \odot \alpha^{(n)} \fCenter \beta$ is analogous and it is omitted). Indeed:

\begin{center}
\AX$\alpha \fCenter \alpha$
\AXC{\fns hyp}
\noLine
\UI$\alpha^{(n)} \odot \beta \fCenter \beta$
\BI$\alpha \odot (\alpha^{(n)} \odot \beta) \fCenter \alpha \cdot \beta$
\UI$(\alpha \odot \alpha^{(n)}) \odot \beta \fCenter \alpha \cdot \beta$
\UI$\alpha^{(n+1)} \odot \beta \fCenter \alpha \cdot \beta$
\AXC{\fns assump}
\noLine
\UI$\alpha \cdot \beta \fCenter \beta$
\RightLabel{\fns cut}
\BI$\alpha^{(n+1)} \odot \beta \fCenter \beta$
\DisplayProof
\end{center}

Hence, the sequent $\alpha^{(n)} \odot \beta \fCenter \beta$ for any $n$ is obtained from a proof of $\alpha \odot \beta \fCenter \beta$ by concatenating $n$ derivations of the shape shown above.
\end{proof}

As to the rule K4  (cf.~Definition \ref{def:KL}), if $\alpha \cdot \beta \fCenter \beta$ is derivable in D.MKL, then $\alpha \odot \beta \fCenter \beta$ is derivable in D.MKL\footnote{This is due to the fact that $\cdot$ is a normal $\mathcal{F}$-operator, and in proper display calculi the left introduction rules of $\mathcal{F}$-operators are invertible.}, hence by Lemma \ref{lemma:OmegaLemma} so are the sequents $\alpha^{(n)} \odot \beta \fCenter \beta $  for any $n\geq 1$.  By applying the appropriate display postulate to each such sequent, we obtain  derivations of $\alpha^{(n)}  \fCenter \beta <\beta $ for any $n\geq 1$.  Hence:

\begin{center}
\AX$(\alpha^{(n)} \fCenter \beta < \beta \,\mid \,1 \leq n)$
\LeftLabel{$\omega$}
\UI$\WCIRC\BCIRC \alpha  \fCenter \beta < \beta$
\UI$\BCIRC \alpha  \fCenter \BCIRC (\beta < \beta)$
\UI$\gtdia \alpha  \fCenter \BCIRC (\beta < \beta)$
\UI$\WCIRC\gtdia \alpha  \fCenter \beta < \beta$
\UI$\tgbox \gtdia \alpha \fCenter \beta < \beta$
\UI$\tgbox \gtdia \alpha \gAND \beta  \fCenter \beta$
\UI$\tgbox \gtdia \alpha \gand \beta  \fCenter \beta$
\DisplayProof
\end{center}

The proof that the rule K5 is derivable is analogous and we omit it. As to the axioms of Definition \ref{def:KL} in which $()^\ast$-terms occur,

\begin{center}
\begin{tabular}{cc}

\AXC{\ }
\LL{\fns one}
\UI$\Phi \fCenter \WCIRC \gtdia \alpha$
\UI$\Phi \fCenter \tgbox \gtdia \alpha$
\UI$1 \fCenter \tgbox \gtdia \alpha$

\AX$\alpha \fCenter \alpha$
\UI$\BCIRC \alpha \fCenter \gtdia \alpha$
\UI$\alpha \fCenter \WCIRC \gtdia \alpha$

\AX$\alpha \fCenter \alpha$
\UI$\BCIRC \alpha \fCenter \gtdia \alpha$
\UI$\gtdia \alpha \fCenter \gtdia \alpha$
\RL{\scriptsize w-bal}
\UI$\WCIRC\gtdia \alpha \fCenter \WCIRC\gtdia \alpha$
\UI$\tgbox\gtdia \alpha \fCenter \WCIRC\gtdia \alpha$

\RL{\fns abs}
\BI$\alpha \odot \tgbox\gtdia \alpha  \fCenter \WCIRC \gtdia \alpha$
\UI$\alpha \odot \tgbox\gtdia \alpha \fCenter \tgbox \gtdia \alpha$
\UI$\alpha \cdot \tgbox\gtdia \alpha \fCenter \tgbox \gtdia \alpha$

\BI$1 \gor \alpha \cdot \tgbox\gtdia \alpha \fCenter \tgbox \gtdia \alpha$
\DisplayProof
 \\
\end{tabular}
\end{center}

\begin{center}
\begin{tabular}{c}
\AXC{\ }
\LL{\fns one}
\UI$\Phi \fCenter \WCIRC \gtdia \alpha$
\UI$\Phi \fCenter \tgbox \gtdia \alpha$
\UI$\gone \fCenter \tgbox \gtdia \alpha$

\AX$\alpha \fCenter \alpha$
\UI$\BCIRC \alpha \fCenter \gtdia \alpha$
\UI$\gtdia \alpha \fCenter \gtdia \alpha$
\RL{\scriptsize w-bal}
\UI$\WCIRC\gtdia \alpha \fCenter \WCIRC\gtdia \alpha$
\UI$\tgbox\gtdia \alpha \fCenter \WCIRC\gtdia \alpha$

\AX$\alpha \fCenter \alpha$
\UI$\BCIRC \alpha \fCenter \gtdia \alpha$
\UI$\gtdia \alpha \fCenter \gtdia \alpha$
\RL{\scriptsize w-bal}
\UI$\WCIRC\gtdia \alpha \fCenter \WCIRC\gtdia \alpha$
\UI$\tgbox\gtdia \alpha \fCenter \WCIRC\gtdia \alpha$

\RL{\fns abs}
\BI$\tgbox\gtdia \alpha \odot \tgbox\gtdia \alpha\fCenter \WCIRC \gtdia \alpha$
\UI$\tgbox\gtdia \alpha \odot \tgbox\gtdia \alpha \fCenter \tgbox \gtdia \alpha$
\UI$\tgbox\gtdia \alpha \cdot \tgbox\gtdia \alpha\fCenter \tgbox \gtdia \alpha$
\BI$\gone \gor \tgbox\gtdia \alpha \cdot \tgbox\gtdia \alpha\fCenter \tgbox \gtdia \alpha$
\DisplayProof
 \\
\end{tabular}
\end{center}

The translations of $0\cdot \alpha\dashv\vdash 0$  are derivable as follows:

\begin{center}
\begin{tabular}{cc}
\AX$0 \fCenter \Phi$
\RL{\fns $\Phi$-W}
\UI$0 \fCenter \Phi < \alpha$
\UI$0 \odot \alpha \fCenter \Phi$
\UI$0 \cdot \alpha \fCenter \Phi$
\UI$0 \cdot \alpha \fCenter 0$
\DisplayProof
 &
\AX$0 \fCenter \Phi$
\RL{\fns $\Phi$-W}
\UI$0 \fCenter \alpha \cdot 0$
\DisplayProof
\end{tabular}
\end{center}

The translation of  $0^* \fCenter 1$  is derivable as follows:

\begin{center}
\AXC{\ }
\LL{\fns one}
\UI$\Phi \fCenter \WCIRC\BCIRC \gone$
\LL{\fns $\Phi$}
\UI$\Phi \gAND \gbot \fCenter\WCIRC\BCIRC \gone$
\LL{\fns $\Phi$}
\UI$\gbot \fCenter \WCIRC\BCIRC \gone$
\UI$\BCIRC \gbot \fCenter \BCIRC \gone$
\UI$\gtdia\gbot \fCenter \BCIRC \gone$
\UI$\WCIRC\gtdia\gbot \fCenter \gone$
\UI$\tgbox\gtdia\gbot \fCenter \gone$
\DisplayProof
\end{center}

The translation of $1 \fCenter 0^\ast$ is derivable as follows:

\begin{center}
\AXC{\ }
\LL{\fns one}
\UI$\Phi \fCenter \WCIRC \gtdia 0$
\UI$\Phi \fCenter \tgbox \gtdia 0$
\UI$1 \fCenter \tgbox \gtdia 0$
\DisplayProof
\end{center}

The translation of $1^\ast \fCenter 1$ is derivable applying the rule $\WCIRC \BCIRC \Phi$ (that is derivable using the $\omega$-rule):

\begin{center}
\AX$\Phi \fCenter 1$
\LL{\fns $\WCIRC \BCIRC \Phi$}
\UI$\WCIRC \BCIRC \Phi \fCenter 1$
\UI$\BCIRC \Phi \fCenter \BCIRC 1$
\UI$\Phi \fCenter \WCIRC \BCIRC 1$
\UI$1 \fCenter \WCIRC \BCIRC 1$
\UI$\BCIRC 1 \fCenter \BCIRC 1$
\UI$\gtdia 1 \fCenter \BCIRC 1$
\UI$\WCIRC \gtdia 1 \fCenter 1$
\UI$\tgbox \gtdia 1 \fCenter 1$
\DisplayProof
\end{center}

The derivations of the translations of the remaining axioms are standard and are omitted.
Below, we derive the translations of the axioms  of Definition \ref{def:DualStar}.

\begin{center}
\begin{tabular}{cc}
\AXC{\ }
\LL{\fns one}
\UI$\Phi \fCenter \WCIRC \gtbox \alpha$
\UI$\Phi \fCenter \tgbox \gtbox \alpha$
\UI$1 \fCenter \tgbox \gtbox \alpha$
\DisplayProof
 &
\AX$\alpha \fCenter \alpha$
\UI$\gtbox \alpha \fCenter \BCIRC \alpha$
\UI$\gtbox \alpha \fCenter \gtbox \alpha$
\RL{\scriptsize w-bal}
\UI$\WCIRC\gtbox \alpha \fCenter \WCIRC\gtbox \alpha$
\UI$\tgbox\gtbox \alpha \fCenter \WCIRC\gtbox \alpha$

\AX$\alpha \fCenter \alpha$
\UI$\gtbox \alpha \fCenter \BCIRC \alpha$
\UI$\gtbox \alpha \fCenter \gtbox \alpha$
\RL{\scriptsize w-bal}
\UI$\WCIRC\gtbox \alpha \fCenter \WCIRC\gtbox \alpha$
\UI$\tgbox\gtbox \alpha \fCenter \WCIRC\gtbox \alpha$

\RL{\fns abs}
\BI$\tgbox\gtbox \alpha \odot \tgbox\gtbox \alpha\fCenter \WCIRC \gtbox \alpha$
\UI$\tgbox\gtbox \alpha \odot \tgbox\gtbox \alpha\fCenter \tgbox \gtbox \alpha$
\DisplayProof
\end{tabular}
\end{center}

\begin{center}
\begin{tabular}{ccc}
\AX$\alpha \fCenter \alpha$
\UI$\gtbox \alpha \fCenter \BCIRC \alpha$
\UI$\WCIRC \gtbox \alpha \fCenter \alpha$
\UI$\tgbox \gtbox \alpha \fCenter \alpha$
\DisplayProof
 &
\AX$\alpha \fCenter \alpha$
\UI$\gtbox \alpha \fCenter \BCIRC \alpha$
\UI$\gtbox \alpha \fCenter \gtbox \alpha$
\LL{\fns b-bal}
\UI$\BCIRC \WCIRC \gtbox \alpha \fCenter \BCIRC \WCIRC \gtbox \alpha$
\UI$\WCIRC \BCIRC \WCIRC \gtbox \alpha \fCenter \WCIRC \gtbox \alpha$
\UI$\WCIRC \BCIRC \WCIRC \gtbox \alpha \fCenter \tgbox \gtbox \alpha$
\UI$\BCIRC \WCIRC \gtbox \alpha \fCenter \BCIRC \tgbox \gtbox \alpha$
\UI$\BCIRC \WCIRC \gtbox \alpha \fCenter \gtbox \tgbox \gtbox \alpha$
\UI$\WCIRC \gtbox \alpha \fCenter \WCIRC \gtbox \tgbox \gtbox \alpha$
\UI$\WCIRC \gtbox \alpha \fCenter \tgbox \gtbox \tgbox \gtbox \alpha$
\UI$\tgbox \gtbox \alpha \fCenter \tgbox \gtbox \tgbox \gtbox \alpha$
\DisplayProof
 &
\AX$\alpha \fCenter \alpha$
\UI$\gtbox \alpha \fCenter \BCIRC \alpha$
\UI$\gtbox \alpha \fCenter \gtbox \alpha$
\RL{\fns w-bal}
\UI$\WCIRC \gtbox \alpha \fCenter \WCIRC \gtbox \alpha$
\UI$\tgbox \gtbox \alpha \fCenter \WCIRC \gtbox \alpha$
\UI$\tgbox \gtbox \alpha \fCenter \tgbox \gtbox \alpha$
\DisplayProof
\end{tabular}
\end{center}

Finally, let us derive the translation of the ternary rule of Definition \ref{def:DualStar}. Assume that the translations of  $\beta \fCenter \alpha$, and
 $\gone \fCenter \beta$ and
$\beta\cdot \beta \fCenter \beta$ are derivable. Hence, by the invertibility of the introduction rules of $\mathcal{F}$-connectives in proper display calculi, $\Phi \fCenter \beta$ and
$\beta\odot \beta \fCenter \beta$ are derivable.   By Lemma \ref{lemma:OmegaLemma},  $\beta^{(n)} \odot \beta \fCenter \beta$ is derivable. Therefore, we can derive the following sequents for any $n\geq 1$:

\begin{center}
\AX$(\beta^{(n)} \fCenter \beta < \beta^{(n)} \,\mid \,1 \leq n)$
\LeftLabel{$\omega$}
\UI$\WCIRC\BCIRC \beta \fCenter \beta < \beta^{(n)}$
\UI$\WCIRC\BCIRC \beta \gAND \beta^{(n)} \fCenter \beta$
\UI$\beta^{(n)} \fCenter \WCIRC\BCIRC \beta > \beta$
\DisplayProof
\end{center}
Hence:

\begin{center}
\AX$(\beta^{(n)} \fCenter \WCIRC\BCIRC \beta > \beta \,\mid \,1 \leq n)$
\LeftLabel{$\omega$}
\UI$\WCIRC\BCIRC \beta \fCenter \WCIRC\BCIRC \beta > \beta$
\UI$\WCIRC\BCIRC \beta \gAND \WCIRC\BCIRC \beta \fCenter \beta$
\LeftLabel{\fns $\WCIRC$-C}
\UI$\WCIRC\BCIRC \beta \fCenter \beta$

\AX$\beta \fCenter \alpha$

\RL{\fns Cut}
\BI$\WCIRC\BCIRC \beta \fCenter \alpha$
\UI$\BCIRC \beta \fCenter \BCIRC \alpha$
\UI$\BCIRC \beta \fCenter \gtbox \alpha$
\UI$\beta \fCenter \WCIRC \gtbox \alpha$
\UI$\beta \fCenter \tgbox \gtbox \alpha$
\DisplayProof
\end{center}

\subsection{Conservativity}
For any heterogeneous measurable Kleene algebra $\mathbb{H} = (\bbA, \bbS, \gamma, \iota, e)$, the algebra $\bbA$ is a  complete join-semilattice, and $\cdot$ distributes over arbitrary joins in each coordinate. This implies that the right residuals exist of $\cdot$ in each coordinate, which we denote $\slash$ and $\backslash$:
$$\alpha \backslash \beta := \gOR\{\alpha': \alpha \cdot \alpha' \leq \beta \}, \quad \beta \slash \alpha := \gOR\{\alpha':  \alpha' \cdot \alpha  \leq \beta \}.$$
From here on, the proof of conservativity proceeds in the usual way as detailed in \cite{GMPTZ}.

\subsection{Cut elimination and subformula property}
The cut elimination of D.MKL follows from  the Belnap-style meta-theorem proven in \cite{Trends}, of which a restriction to  proper multi-type display calculi is stated in \cite{GP:linear}. The proof boils down to verifying the conditions $C_1$-$C_{10}$ of \cite[Section 6.4]{GP:linear}. Most of these conditions are easily verified by inspection on rules; the most interesting one is condition $\textrm{C}'_8$, concerning the principal stage in the cut elimination, on which we expand in the lemma below.
\begin{lemma}\label{cut:C8}
$\mathrm{D.MKL}$ satisfies $\textrm{C}'_8$.
\end{lemma}
\begin{proof}
By induction on the shape of the cut formula.
\paragraph*{Atomic propositions:}

\begin{center}
\begin{tabular}{ccc}
\bottomAlignProof
\AX$a \fCenter a$
\AX$a \fCenter a$
\BI$a \fCenter a$
\DisplayProof

 & $\rightsquigarrow$ &

\bottomAlignProof
\AX$a \fCenter a$
\DisplayProof
 \\
\end{tabular}
\end{center}

\paragraph*{Constants:}

\begin{center}
\begin{tabular}{ccc}
\bottomAlignProof
\AX$\gONE \fCenter \gone$
\AXC{\ \ \ $\vdots$ \raisebox{1mm}{$\pi_1$}}
\noLine
\UI$\gONE \fCenter \Delta$
\UI$\gone \fCenter \Delta$
\BI$\gONE \fCenter \Delta$
\DisplayProof

 & $\rightsquigarrow$ &

\bottomAlignProof
\AXC{\ \ \ $\vdots$ \raisebox{1mm}{$\pi_1$}}
\noLine
\UI$\gONE \fCenter \Delta$
\DisplayProof
 \\
\end{tabular}
\end{center}

\noindent The cases for $\gbot$, $\tbot$ are standard and similar to the one above.

\paragraph*{Unary connectives:}
As to $\gtdia\alpha$,
\begin{center}
\begin{tabular}{ccc}
\!\!\!\!\!
\bottomAlignProof
\AXC{\ \ \ $\vdots$ \raisebox{1mm}{$\pi_1$}}
\noLine
\UI$\Gamma \fCenter \alpha$
\UI$\BCIRC \Gamma \fCenter \gtdia\alpha$

\AXC{\ \ \ $\vdots$ \raisebox{1mm}{$\pi_2$}}
\noLine
\UI$\BCIRC \alpha \fCenter \Xi$
\UI$\gtdia\alpha \fCenter \Xi$
\BI$\BCIRC \Gamma \fCenter \Xi$
\DisplayProof

 & $\rightsquigarrow$ &

\!\!\!\!\!\!\!\!\!\!\!\!\!\!\!\!\!\!\!\!
\bottomAlignProof
\AXC{\ \ \ $\vdots$ \raisebox{1mm}{$\pi_1$}}
\noLine
\UI$\Gamma \fCenter \alpha$
\AXC{\ \ \ $\vdots$ \raisebox{1mm}{$\pi_2$}}
\noLine
\UI$\BCIRC \alpha \fCenter \Xi$
\UI$\alpha \fCenter \WCIRC\Xi$
\BI$ \Gamma \fCenter \WCIRC\Xi$
\UI$\BCIRC \Gamma \fCenter \Xi$
\DisplayProof
 \\
\end{tabular}
\end{center}

As to $\tgbox\alpha$,

\begin{center}
\begin{tabular}{ccc}
\!\!\!\!\!
\bottomAlignProof
\AXC{\ \ \ $\vdots$ \raisebox{1mm}{$\pi_1$}}
\noLine
\UI$\Gamma \fCenter \WCIRC\xi$
\UI$\Gamma \fCenter \tgbox\xi$

\AXC{\ \ \ $\vdots$ \raisebox{1mm}{$\pi_2$}}
\noLine
\UI$\xi \fCenter \Xi$
\UI$\tgbox\xi \fCenter \WCIRC\Xi$
\BI$\Gamma \fCenter \WCIRC\Xi$
\DisplayProof

 & $\rightsquigarrow$ &

\!\!\!\!\!\!\!\!\!\!\!\!\!\!\!\!\!\!
\bottomAlignProof
\AXC{\ \ \ $\vdots$ \raisebox{1mm}{$\pi_1$}}
\noLine
\UI$\Gamma \fCenter \WCIRC\xi$
\UI$\BCIRC \Gamma \fCenter \xi$

\AXC{\ \ \ $\vdots$ \raisebox{1mm}{$\pi_2$}}
\noLine
\UI$\xi \fCenter \Xi$
\BI$\BCIRC \Gamma \fCenter \Xi$
\UI$\Gamma \fCenter \WCIRC\Xi$
\DisplayProof
 \\
\end{tabular}
\end{center}

\paragraph*{Binary connectives:}
As to $\alpha_1 \gor \alpha_2$,
\begin{center}
\begin{tabular}{ccc}
\!\!\!\!\!
\bottomAlignProof
\AXC{\ \ \ $\vdots$ \raisebox{1mm}{$\pi_1$}}
\noLine
\UI$\Gamma \fCenter \alpha_1$
\UI$ \Gamma \fCenter \alpha_1 \gor \alpha_2$
\AXC{\ \ \ $\vdots$ \raisebox{1mm}{$\pi_2$}}
\noLine
\UI$\alpha_1 \fCenter \Delta$
\AXC{\ \ \ $\vdots$ \raisebox{1mm}{$\pi_3$}}
\noLine
\UI$\alpha_2 \fCenter \Delta$
\BI$\alpha_1 \gor \alpha_2 \fCenter \Delta$
\BI$\Gamma \fCenter \Delta$
\DisplayProof

 & $\rightsquigarrow$ &

\!\!\!\!\!\!\!\!\!\!\!\!\!\!\!\!\!\!\!\!
\bottomAlignProof
\AXC{\ \ \ $\vdots$ \raisebox{1mm}{$\pi_1$}}
\noLine
\UI$\Gamma \fCenter \alpha_1$
\AXC{\ \ \ $\vdots$ \raisebox{1mm}{$\pi_2$}}
\noLine
\UI$\alpha_1 \fCenter \Delta$
\BI$\Gamma \fCenter \Delta$
\DisplayProof
 \\
\end{tabular}
\end{center}
\end{proof}

\bibliography{dissertation-5}
\bibliographystyle{plain}

\end{document}

%% file: packages.tex
\usepackage{lscape}
\usepackage{pdflscape}

\usepackage{stmaryrd}

\usepackage{cmll}

\usepackage{comment}

\usepackage{amscd}
\usepackage{amssymb,amsmath,amsthm}
\usepackage{mathptmx}
\usepackage{mathrsfs}
\usepackage{color}
\usepackage{xspace}
\usepackage{bussproofs}
\EnableBpAbbreviations

\usepackage{tikz}

  \usepackage{txfonts}

\usepackage{bpextra}

\usepackage{enumerate}

\usepackage{centernot}

\usepackage{mathtools}

\usepackage{hyperref}

\usepackage{cleveref}

\usepackage{relsize}

\usepackage{caption}

\usepackage{multirow}

\usepackage{multicol}

\usepackage{array}
\newcolumntype{C}[1]{>{\centering\arraybackslash}p{#1}}
\newcolumntype{L}[1]{>{\arraybackslash}p{#1}}

%% file: macros2.tex
\newcommand{\bbA}{\mathbb{A}}
\newcommand{\bbS}{\mathbb{s}}

\newcommand{\gONE}{\Phi}
\newcommand{\gone}{1}

\newcommand{\gBOT}{\Phi}
\newcommand{\gbot}{0}

\newcommand{\gOR}{\,\bigcup\,}
\newcommand{\gor}{\cup}

\newcommand{\gAND}{\odot}
\newcommand{\gand}{\cdot}

\newcommand{\tone}{1_s}

\newcommand{\tbot}{0_s}

\newcommand{\tor}{\sqcup}

\newcommand{\mand}{\ensuremath{\otimes}\xspace}

\newcommand{\WCIRC}{\ensuremath{\circ}\xspace}
\newcommand{\BCIRC}{\ensuremath{\bullet}\xspace}
\newcommand{\tgbox}{\ensuremath{\Box}\xspace}

\newcommand{\gtbox}{\ensuremath{{\scriptstyle\blacksquare}}\xspace}
\newcommand{\gtdia}{\ensuremath{\blacklozenge}\xspace}